\documentclass[11pt,reqno]{amsart} 
\usepackage[left=3.2cm,right=3.2cm,bottom=4.8cm]{geometry}
\usepackage{amsfonts}
\usepackage{amsmath}
\usepackage{amssymb}
\usepackage{mathtools}
\usepackage{amsthm}
\usepackage{pgfplots}
\pgfplotsset{compat=newest}
\usepackage[shortlabels]{enumitem}
\usepackage{tikz}   
\usepackage{tikz-cd}  
\usepackage{changepage}
\usetikzlibrary{arrows}
\usepackage{bm}
\usepackage{graphicx}
\usepackage{cases}
\usepackage{appendix}
\usepackage{esint}
\usepackage{hyperref}
\usepackage{mathrsfs}
\usepackage{mathrsfs}
\numberwithin{equation}{section}

\theoremstyle{definition}
\newtheorem{thm}{Theorem}
\newtheorem*{claim}{Claim}
\newtheorem{defn}[thm]{Definition}
\newtheorem{lem}[thm]{Lemma}
\newtheorem{prop}[thm]{Proposition}
\newtheorem{cor}[thm]{Corollary}
\newtheorem*{rmk}{Remark}
\newtheorem*{ex}{Example}
\newtheorem*{question}{Question}
\numberwithin{thm}{section}

\newcommand{\R}{\mathbb{R}}  

\newcommand{\p}{\partial}  

\newcommand{\dif}{\textup{d}} 

\newcommand{\dist}{\textup{dist}}

\newcommand{\loc}{\textup{loc}}

\usepackage{xcolor}
\usepackage{listings}
\begin{document}
\title[Expanding harmonic map flows]{Existence of expanding harmonic map flows to hemispheres} 
\author[Xuanyu Li]{Xuanyu Li}
\address{Department of Mathematics, Cornell University, Ithaca, NY 14853, USA}
\email{xl896@cornell.edu}

\begin{abstract}
    We show the existence of non-trivial self-expanding harmonic map flows starting from non-energy-minimizing 0-homogeneous maps to a regular ball or a closed hemisphere. In particular, given a non-minimizing but stationary 0-homogeneous harmonic map $u_0$ to a closed hemisphere, we construct infinitely many different weak solutions to harmonic map flow starting from $u_0$, all of which satisfy the parabolic monotonicity formula. We also use a concatenation argument to solve a non-trivial smooth harmonic map flow starting from a given non-minimizing regular 0-homogeneous harmonic map, thereby yielding the non-uniqueness of harmonic map flow with monotonicity formula. This answers a question of Struwe.
\end{abstract}

\maketitle

\section{introduction}
The harmonic map flow arises naturally from the variational theory of energy functional. For maps from $\R^n$ to a closed Riemann manifold $M^m$, the energy functional is defined as
$$E_{\Omega}(v)=\int_{\Omega}\vert\nabla v\vert^2\dif x,\text{ for }v\in W^{1,2}_{\loc}(\R^n,M).$$
Harmonic map flow is formally the $L^2$ gradient flow of this energy functional. For simplicity we assume $M\subset\R^L$ is isometrically embedded in some large Euclidean space, with second fundamental form $A$. Given an initial map $u_0\in W^{1,2}_{\loc}(\R^n,M)$, we say $u\in C^{\infty}(\R^n\times(0,T),M)$ is the \textbf{harmonic map flow} starting from $u_0$, if $u$ satisfies
\begin{itemize}
    \item $\p_t u=\Delta u+A_u(\nabla u,\nabla u)$ in $\R^n\times(0,T)$;
    \item $u=u_0$ on $\R^n\times\lbrace0\rbrace$ in the sense of trace;
    \item $u(\cdot,t)\rightarrow u_0$ strongly in $W^{1,2}_{\loc}$ as $t\rightarrow0$.
\end{itemize}

As a nonlinear evolution problem, the uniqueness of the harmonic map flow does not follow directly from a priori estimates and has therefore been extensively studied. In dimension two, Coron \cite{Coron} and Bethuel, Coron and Ghidaglia \cite{BethuelCoronGhidaglia} first constructed the non-unique weak solutions. However, none of the weak solutions obtained in these works satisfies either the energy monotonicity formula or the parabolic stationarity condition. Moreover, Freire \cite{Freire}, Rupflin \cite{Rupflin} and L. Wang \cite{LuWang} showed that, in two dimensions, harmonic map flows satisfying the energy monotonicity property are unique. Motivated by these results, Struwe raised the following question; see \cite[p. 291]{StruweGeometricEvolution} or \cite[Section 2.4]{StruweNormalizedFlow}.
\begin{question}[Struwe]
    In arbitrary dimensions, suppose that $u_1$ and $u_2$ are weak solutions to the harmonic map flow with the same initial map. If both solutions satisfy the energy monotonicity formula, or the parabolic stationary condition introduced by Feldman \cite{Feldman}, must it follow that $u_1=u_2$? 
\end{question}

We note that Ilmanen also raised the problem of non-uniqueness for the harmonic map flow starting from singular initial data (see the unpublished discussion cited in \cite[Lecture 4]{IlmanenLectureNotes}). We refer the reader to Section~\ref{sec: preliminaries} for the precise statement of the energy monotonicity formula.

Particular counterexamples to Struwe’s question have been constructed in equivariant settings by M. Hong \cite{Hong}, Germain and Rupflin \cite{GermainRupflin}, and Pierre Germain, Ghoul and Miura \cite{GermainGhoulMiura}. However, the general case remains largely open. In this paper, we provide a negative answer to Struwe’s question for a broad class of harmonic map flows by establishing the following existence results.

\begin{thm}\label{thm: main theorem 2}
    Suppose $3\leqslant n\leqslant6$. For any regular 0-homogeneous map $u_0:\R^n\rightarrow\overline{S^m_+}$, there exists an expanding harmonic map flow $u\in C^{\infty}(\R^n\times[0,\infty)\setminus\lbrace(0,0)\rbrace,\overline{S^m_+})$ starting from $u_0$.
\end{thm}

The original question of Struwe ask the non-uniqueness of harmonic map flow whose domain is compact. All the domains for all previous examples are Euclidean spaces, which are non-compact. We also show via a concatenation argument that Theorem \ref{thm: main theorem 2} implies the non-uniqueness in the bounded domain.
\begin{cor}\label{thm: non-uniqueness on bounded domain}
    Suppose $3\leqslant n\leqslant 6$. For any regular 0-homogeneous map $u_0:\R^n\rightarrow\overline{S^m_+}$, there exists $T>0$ and a harmonic map flow $u\in C^{\infty}(\bar B_1\times(0,T],\overline{S^m_+})$ starting from $u_0$ and satisfying the boundary condition $u(\cdot,t)|_{\p B_1}=u_0|_{\p B_1}$ for all $t\in[0,T]$.
\end{cor}

For the case $n\geqslant7$, we are still able to get nontrivial harmonic map flow starting from a map that is not energy minimizing, although a singular set of codimension 7 may occur:
\begin{thm}\label{thm: main theorem 3}
    Suppose $n\geqslant7$. For any regular $0$-homogeneous map $u_0:\R^n\rightarrow \overline{S^m_+}$ that is not energy minimizing with respect to the standard Euclidean metric, there exists a non-trivial expanding harmonic map flow $u\in W^{1,2}_{\loc}(\R^n\times(0,\infty),\overline{S^m_+})$ starting from $u_0$ which is smooth near $(\R^n\setminus\lbrace0\rbrace)\times\lbrace0\rbrace$. Moreover, $u(\cdot,1)$ viewed as a map from $\R^n$ to $\overline{S^m_+}$ is regular away from a compact set $\Sigma$ with Hausdorff dimension no greater than $n-7$. When $n=7$, $\Sigma$ is discrete. In particular, $u$ satisfies the parabolic stationary condition.
\end{thm}

In particular, if $u_0$ is not a stationary harmonic map, then it is not energy minimizing; see F. Lin and C. Wang \cite{LinWangBook} for the definition and properties of stationary harmonic maps. On the other hand, a stationary 0-homogeneous harmonic map may be viewed as a static solution to the harmonic map flow, and hence satisfies the corresponding parabolic monotonicity formula as well as the parabolic stationary condition. For $3\leqslant n\leqslant 6$, such a map cannot be energy minimizing by the regularity result of Schoen and Uhlenbeck \cite{SchoenUhlenbeckSphere}. Moreover, there exist many nontrivial stationary 0-homogeneous harmonic maps into a closed hemisphere. Consequently, Theorem \ref{thm: main theorem 2} and Theorem \ref{thm: main theorem 3} may be interpreted as non-uniqueness results for harmonic map flow with such initial data. Related studies of expanding harmonic map flow starting from 0-homogeneous initial maps were carried out by Deruelle and Lamm \cite{DeruelleLamm}, Deruelle \cite{Deruelle}, and Z. Geng, C. Wang and J. Yu \cite{GengWangYu}. However, in these works either the initial maps cannot be nontrivial harmonic maps, or the resulting flows cannot be shown to be nontrivial when the initial map itself is harmonic.

Here, we denote by $\overline{S^m_+}$ the closed $m$-dimensional upper hemisphere and by $S^m_+$ its interior. We say a map $u_0:\R^n\rightarrow M$ is a \textbf{regular 0-homogeneous map}, if $u_0$ is smooth away from the origin and $u_0(x)=u_0(\lambda x)$ for all $x\in \R^n$ and $\lambda>0$. Such maps arise naturally as singularity models for harmonic maps via the dimension reduction principle introduced by Federer \cite{Federer}. For scaling-invariant initial data of this type, solutions of the harmonic map flow that are of particular interest are the self-similar ones, namely those invariant under the parabolic scaling
$$u(x,t)=u(\lambda x,\lambda^2 t),\text{ for all }x\in\R^n, t>0\text{ and }\lambda>0.$$
Such solutions are called (self-)expanding solutions of the harmonic map flow.

We now explain the derived non-uniqueness. As observed by Coron \cite{Coron}, for each stationary harmonic 0-homogeneous map $u_0$, once we have a non-trivial solution $u$ to harmonic map flow starting from $u_0$, we can define a new solution to harmonic map flow with initial data $u_0$ by time translation. More precisely, for any $T>0$, define
\begin{align*}
    \bar{u}(\cdot,t)=
    \begin{cases} 
        u_0,&\text{ when }t\leqslant T;\\
        u(\cdot,t-T),&\text{ when }t> T.
    \end{cases}
\end{align*}
As a consequence, Theorem \ref{thm: main theorem 2} and \ref{thm: main theorem 3} yield uncountably many distinct solutions of the harmonic map flow starting from $u_0$, all of which satisfy the energy monotonicity formula. We also remark that the smooth solution given by Theorem \ref{thm: main theorem 2} is not necessarily unique, see the following example.
\begin{ex}
    Let $u_0':\R^n\rightarrow S^{n-1}$ be a regular 0-homogeneous harmonic map such that $u_0'|_{\p B_1}: S^{n-1}\rightarrow S^{n-1}$ has non-zero topology degree. View $S^{n-1}$ as the equator of $S^n$, and set $u_0=(u_0',0)$. For $3\leqslant n\leqslant 6$, let $u$ be the solution given by Theorem \ref{thm: main theorem 2}. We claim that $u$ must take the form $(u',u_{n+1})$ for a function $u_{n+1}>0$ in $\R^n\times(0,\infty)$. Indeed, for the sphere target harmonic map flow, the equation takes the form
    $$\p_t u_{n+1}-\Delta u_{n+1}=\vert\nabla u\vert^2 u_{n+1}\geqslant0.$$
    By maximum principle for subsolutions to heat equation, we have either $u_{n+1}>0$ or $u_{n+1}=0$ in $\R^n\times(0,\infty)$. If $u_{n+1}=0$ in $\R^n\times(0,\infty)$, we have that $u'(\omega,t)$ is a homotopy between two $S^{n-1}$-valued maps $u_0'|_{\p B_1}$ and $u'(\cdot,1)|_{\p B_1}$. Therefore, $u'(\cdot,1)|_{\p B_1}:S^{n-1}\rightarrow S^{n-1}$ is a map with non-trivial topology degree. On the other hand, $u'(\cdot,1)\in C^{\infty}(\bar{B}_1,S^{n-1})$ smoothly extends $u'(\cdot,1)|_{\p B_1}$, which is a contradiction because this implies that the topology degree of $u'$ is 0. 

    Hence $u_{n+1}>0$ in $\R^n\times(0,\infty)$. Now, if we embed $S^n$ further to be the equator of $S^{n+1}$, and consider the initial map $\tilde{u}_0=(u_0',0,0)$, then for any $\theta\in(0,\pi/2)$, $\tilde{u}^{\theta}=(u',u_{n+1}\cos\theta,u_{n+1}\sin\theta)$ is an expanding harmonic map flow starting from $\tilde{u}_0$. This tells us that the smooth solution given by Theorem \ref{thm: main theorem 2} is not unique in general.
    
    We conjecture that, if the image of $u_0:\R^n\rightarrow S^m$ is not contained in any m-dimensional linear subspace of $\R^{m+1}$, then the smooth expanding harmonic map flow starting from $u_0$ is unique. For the minimizing harmonic maps from a bounded domain to hemispheres, this was done by Sandier-Shafrir \cite{SandierShafrir}.
\end{ex}

The proofs of Theorem \ref{thm: main theorem 2} and Theorem \ref{thm: main theorem 3} proceed by perturbing $u_0$ to $u_0^{\sigma}$ whose image is contained in a compact subset of the open hemisphere $S^m_+$. One then constructs an expanding solution to the harmonic map flow starting from $u_0^{\sigma}$ and finally lets $\sigma\rightarrow0$. It is known that any compact subset of $S^m_+$ is contained in a regular ball of $S^m$. Recall that, as introduced by Hildebrandt-Kaul-Widman \cite{HildebrandtStefanWidman}, we say a geodesic ball $B_{\rho}(p)$ in a Riemannian manifold $M$ is a \textbf{regular ball} if
\begin{itemize}
    \item $\rho<\frac{\pi}{2\sqrt{\kappa}}$ where $\kappa\geqslant0$ is an upper bound of $\sec_M$;
    \item $B_{\rho}(p)$ does not intersect the cut locus of $p$.
\end{itemize}

\begin{thm}\label{thm: main theorem 1}
    Let $B_{\rho}(p)\subset M$ be a regular ball. For any regular 0-homogeneous map $u_0:\R^n\rightarrow B_{\rho}(p)$, there exists a unique smooth harmonic map flow $u\in C^{\infty}(\R^n\times(0,\infty),B_{\rho}(p))$ starting from $u_0$ which is expanding.
\end{thm}
When $M$ has non-positive sectional curvature, $M$ itself is a regular ball. In this setting, Deruelle \cite{Deruelle} established the existence of expanding solutions to the harmonic map flow. Theorem \ref{thm: main theorem 1} may therefore be viewed as a generalization of Deruelle’s result to a broader class of target manifolds.

In the sequel, when there is no ambiguity we identify a self-similar flow $u$ with its time-one slice $u(x)=u(x,1)$. By self-similar condition, we have that $u(x,t)=u(x/\sqrt t)$. Substituting this ansatz into the harmonic map flow equation transforms the parabolic problem into the following elliptic problem:
\begin{itemize}
    \item $\Delta u+\nabla_{\frac{x}{2}}u+A_{u}(\nabla u,\nabla u)=0$, in $\R^n$;
    \item $u(r\omega)\rightarrow u_0(\omega)$ for all $\omega\in S^{n-1}$, as $r\rightarrow\infty$.
\end{itemize}
We call the solutions to this equation \textbf{(self)-expander}s. Note that this is just the equation for harmonic maps from $(\R^n,e^{\frac{\vert x\vert^2}{2(n-2)}}\delta)$ to $M$, where $\delta$ is the standard Euclidean metric on $\R^n$. In particular, expanders arise as local critical points of an associated weighted energy functional

$$\mathbf{E}(v)=\int\vert\nabla v\vert^2e^{\frac{\vert x\vert^2}{4}}\dif x.$$

The expanders obtained in Theorem \ref{thm: main theorem 3} are local minimizers of this $\mathbf{E}$ functional. As a result, they satisfy the stationary condition as a harmonic map. When interpreted as solutions to the harmonic map flow, this stationarity corresponds precisely to the parabolic stationary condition.
\begin{rmk}
    The smoothness assumption on the initial map $u_0$ for Theorems \ref{thm: main theorem 2}-\ref{thm: main theorem 1} is not essential for the existence of nontrivial expanding solutions to the harmonic map flow starting from $u_0$. Indeed, if $u_0$ is a general $W^{1,2}_{\loc}$ 0-homogeneous map, we can use the theorem of Bethuel \cite{BethuelApproximation} to approximate $u_0$ by regular $0$-homogeneous maps and argue as Proposition \ref{prop: existence} below. The corresponding existence results then follow from the compactness of energy minimizing harmonic maps. When $u_0$ is merely a $W^{1,2}_{\loc}$ map, in general it is not possible to get the smoothness of the resulting flow near $(\R^n\setminus\lbrace 0\rbrace)\times \lbrace0\rbrace$, nor the boundedness of singular set $\Sigma$ in Theorem \ref{thm: main theorem 3}.
\end{rmk}

\subsection{Related history on existence and uniqueness of harmonic map flows}

The existence and uniqueness for harmonic map flows have been extensively studied. The first remarkable result in this direction is the work of Eells and Sampson \cite{EellsSampson} who proved the existence and uniqueness of global smooth harmonic map flow into a non-positively curved manifold, starting from arbitrary $C^2$ initial data. Jost \cite{Jost} derived a similar result to a regular ball, based on the maximum principle by J\"ager and Kaul \cite{JagerKaul}. For general targets and when $n=2$, Struwe \cite{StruweSurface} and K. C. Chang \cite{Chang} proved the existence of solutions that are smooth away from finitely many points. In all dimensions, Y. M. Chen-Struwe \cite{Chen,ChenStuwe} and Y. M. Chen-F. Lin \cite{ChenLin} constructed partially regular solutions. Note that their results even apply to the $W^{1,2}$ initial data, which is called rough initial data. See also the very recent work for harmonic map flow into CAT(0) spaces by F. Lin, Segatti, Sire and C. Wang \cite{lin2026heatflowharmonicmaps}.

When $n=2$, the uniqueness of harmonic map flow is well-understood, even from rough initial data. If the energy is non-increasing in time, Freire \cite{Freire}, Rupflin \cite{Rupflin} and L. Wang \cite{LuWang} showed the uniqueness. But it turns out the uniqueness cannot hold without the energy non-increasing condition, see the work of Bertsch-Dal Passo-van der Hout \cite{BertschDalPassoRoberta} and Topping \cite{Topping}. For higher dimensions, as already discussed, the uniqueness cannot hold for rough initial data even when the energy is nonincreasing. Coron \cite{Coron}, Bethuel-Coron-Ghidaglia \cite{BethuelCoronGhidaglia} and M. Hong \cite{Hong} have constructed specific examples for non-uniqueness. Uniqueness until first singular time is however true for initial maps that are regular enough or a suitable smallness condition is valid, see the work of F. Lin and C. Wang \cite{LinWangUniqueness1} and T. Huang and C. Wang \cite{HuangWang}.

There are also some works on the self-similar harmonic map flows. Special solutions were found using ODE method in equivariant settings; see H. Fan \cite{Fan} and Gastel \cite{Gastel} for shrinkers (backward self-similar flows), Germain and Rupflin \cite{GermainRupflin} and Germain, Ghoul and Miura \cite{GermainGhoulMiura} for expanders. For general initial 0-homogeneous maps, Deruelle and Lamm \cite{DeruelleLamm} constructed the weak expanding flows if the initial data is homotopically trivial. Smooth flows were also constructed when target is a non-positively curved manifold by Deruelle \cite{Deruelle} or the initial data satisfies the smallness assumption by Z. Geng, C. Wang and J. Yu \cite{GengWangYu}. The space of smooth expanders has also been studied in \cite{Deruelle}. We also refer the reader to the works on expanders of other geometric flows \cite{Jiasverak,Deruelle2,Ding,BernsteinWang1,BernsteinWang2,Khan,Wang}

\subsection{Sketch of proofs}
When $u_0$ is a map to a regular ball, we approach this problem directly, by first establishing the existence of harmonic map flow starting from any $u_0\in W^{1,2}_{\loc}(\R^n,B_{\rho}(p))$. When $u_0$ is restricted to a regular 0-homogeneous map, by using monotonicity formula and small energy regularity, for a solution of harmonic map flow $u$ in this case, we establish the estimate
$$\vert\nabla u(x,t)\vert+\sqrt{t}\vert \p_t u(x,t)\vert\leqslant C\vert x\vert,\text{ when }\vert x\vert^2>>t.$$
The gradient estimate together with the maximum principle by J\"ager-Kaul allow us to establish the uniqueness of harmonic map flow from 0-homogeneous data. As a result, such a flow must be self-similar since $u_{\lambda}=u(\lambda\cdot,\lambda^2\cdot)$ is another solution starting from the same initial data. Existence of harmonic map flow is relatively straightforward as those flows into a regular ball enjoy good a priori estimates.

For initial value lying in a closed hemisphere, we note that any compact subset in an open hemisphere is contained in a regular ball. We perturb the initial value a little bit to a regular ball of $S^m$ and solve the corresponding expanding harmonic map flow by Theorem \ref{thm: main theorem 1}. The resulting expanders are energy minimizers with respect to a metric $g$ on $\R^n$, and are bounded in $W^{1,2}_{\loc}$ using monotonicity formula. Such energy minimizing harmonic maps have a priori good compactness results by Schoen-Uhlenbeck \cite{SchoenUhlenbeckRegularity}, which allows us to take a limit to deduce the desired existence of harmonic map flow. The regularity is then already known by Schoen-Uhlenbeck \cite{SchoenUhlenbeckSphere}.

Finally, from a smooth expander $u$ starting from a stationary regular 0-homogeneous map $u_0$, we construct a harmonic map flow $v\in C^{\infty}(\bar B_1\times (0,T],S^k)$. The smoothness guarantees the nontriviality of $v$. We look at the concatenation $\hat u$ between $u$ and $u_0$
$$\hat u=\frac{\zeta u+(1-\zeta)u_0}{\vert \zeta u+(1-\zeta)u_0\vert}.$$
The asymptotics derived by Deruelle \cite{Deruelle} shows that $u(x,t)$ converges exponentially to $u_0$ on $B_1\setminus B_{1/3}$ as $t\rightarrow0$. Hence, $\hat u$ is not far away from being a harmonic map flow. We solve a harmonic map flow $v$ of the form $v=\hat u+w$. The equation to be satisfied by $w$ is 
$$\p_t w=\Delta w+\mathcal{F}(w), \text{ where }\mathcal{F}(w)=-\mathcal{E}(\hat u)+\vert\nabla\hat u\vert^2w+(\vert\nabla w\vert^2+2\langle\nabla\hat u,\nabla w\rangle)(\hat u+w).$$
Given the asymptotics of $u$, we establish good parabolic estimates of $\mathcal{F}$ on certain Banach space. The existence of $w$, hence $v$, is then an application of the contraction mapping theorem.

\subsection{Organization of paper}
In Section \ref{sec: preliminaries} we recall several properties that will be used in this paper regarding the harmonic maps and their heat flows into a regular ball and the energy minimizing maps. We also establish the existence of harmonic map flow to a regular ball starting from $W^{1,2}_{\loc}$ maps. In Section \ref{sec: proofs} we use them to prove our main results.

\subsection*{Acknowledgment}
The author is highly grateful to Prof. Xin Zhou and Prof. Daniel Stern for their many valuable discussions and insights and constant support. The author also wants to thank Zhihan Wang for many fruitful conversations and pointing out the related results in the area of mean curvature flow. The author was partially supported by NSF DMS-2243149 and NSF DMS-2404992.

\section{Preliminaries}\label{sec: preliminaries}
In this section we recall some basics on harmonic maps and their heat flows. We refer readers to the book by F. Lin- C. Wang \cite{LinWangBook} for a more comprehensive account of these objects. We start with the following monotonicity formulas and small energy regularity theorems.
In this section $\Omega\subset\R^n$ denotes a smooth bounded domain. We may equip $\Omega$ with a metric $g$ that differs from the standard one on $\R^n$ and denote it $(\Omega,g)$. When we simply write $\Omega$ or $\R^n$ without having a metric here, we are referring to the standard Euclidean metric on these spaces. Recall that $\kappa$ is an upper bound of the sectional curvature $\sec_M$.

\begin{prop}\label{prop: monotonocity}
Let $\kappa'=\sup\vert\sec_g\vert\geqslant0$. There exists a small $\epsilon_0=\epsilon_0(m,\kappa,\kappa')>0$ with the following property.
    \begin{enumerate}
        \item (Schoen-Uhlenbeck \cite{SchoenUhlenbeckRegularity}) Let $u\in C^{\infty}((\Omega,g),M)$ be a harmonic map. 
        Then for some $C=C(n,\kappa')>0$, the following quantity
        $$\frac{e^{Cr^2}}{r^{n-2}}\int_{B_r^g}\vert\nabla u\vert^2_g\,\dif\text{vol}_g,$$
        is increasing in $r$. And if $r^{2-n}\int_{B_r}\vert\nabla u\vert^2\leqslant\epsilon_0$, then
        $$\sup_{B_{r/2}^g}r^2\vert\nabla u\vert^2\leqslant\frac{C}{r^{n-2}}\int_{B_r^g}\vert\nabla u\vert^2_g\,\dif\text{vol}_g.$$
        \item (Struwe \cite{StruweHigherDimensions}) Let $u\in C^{\infty}(\Omega\times [-T,0],M)$ be a harmonic map flow with $E_{\Omega}(u(t))\leqslant E_0$ for all $t\in[-T,0]$. Fix $R>0$ and a cut-off function $\varphi\in C_0^{\infty}(B_R)$ with $0\leqslant \varphi\leqslant1$ and $\varphi=1$ on $B_{R/2}$. Then for some $C=C(n)>0$,
        \[
        \begin{aligned}
             r_1^2\int_{\Omega\times\lbrace{-r_1^2\rbrace}}\vert\nabla u\vert^2G\varphi^2\,\dif x 
             \leqslant e^{C(r_2-r_1)}r_2^2\int_{\Omega\times\lbrace{-r_2^2\rbrace}}\vert\nabla u\vert^2G\varphi^2\,\dif x+ CE_0(r_2-r_1).
        \end{aligned}
        \]
        Here $G(x,t)=\frac{1}{(-4\pi t)^{n/2}}e^{\vert x\vert^2/4t},t<0$ is the backward heat kernel. And there exists $\delta=\delta(n,\kappa, E_0)>0$ such that if for some $0<r<\sqrt{T}/2$ we have $r^2\int_{\Omega\times\lbrace-r^2\rbrace}\vert\nabla u\vert^2G\varphi^2\,\dif x\leqslant\epsilon_0$, then
        $$\sup_{B_{\delta r}\times[-\delta^2r^2,0]}(r^2\vert\nabla u\vert^2+r^4\vert\p_tu\vert^2)\leqslant C.$$
        Moreover, for some smaller $\epsilon_0'=\epsilon_0'(m,\kappa,T)\leqslant\epsilon_0$, if 
        \[ r^2\int_{\Omega\times\lbrace-r^2\rbrace}\vert\nabla u\vert^2G\varphi^2\,\dif x\leqslant\epsilon_0', \]
        then
        $$\sup_{B_{r/2}\times[-r^2/4,0]}(r^2\vert\nabla u\vert^2+r^4\vert\p _tu\vert^2)\leqslant C.$$
    \end{enumerate}
\end{prop}
The proof is relatively easy for the smooth harmonic maps and smooth harmonic map flows which only requires a blow-up technique by Schoen, see \cite[Chapter IX]{SchoenYau}, see also Struwe \cite{StruweHigherDimensions} for the flow case. The bound on $\vert \p_t u\vert$ needs a little more work, see \cite[Proposition 7.1.4]{LinWangBook}. The proof in the stationary case is however more involved, see \cite{Evans,Bethuel,ChangWangYang,Liu,ChenStuwe} for related discussions.

\subsection{Harmonic map and their heat flow into regular balls}
The harmonic map and their heat flows enjoy much better behavior when the target is restricted to a regular ball. In this subsection we fix a regular ball $B_{\rho}(p)\subset M^m$, with $\rho<\frac{\pi}{2\sqrt{\kappa}},\kappa=\sup\sec_M\geqslant0$. 

We have the following maximum principle.
\begin{prop}[{J\"ager-Kaul \cite{JagerKaul}}]\label{prop: uniqueness}
    Let $\Omega\subset\R^n$ be a bounded smooth domain. Given two maps $u_1,u_2$, denote by $\theta(u_1,u_2)$ the function 
    $$x\mapsto\frac{q_{\kappa}(\dist(u_1(x),u_2(x)))}{\cos(\sqrt{\kappa}\dist(u_1(x),p))\cos(\sqrt{\kappa}\dist(u_2(x),p))},$$
    where 
    \begin{align*}
        q_{\kappa}(t)=\begin{cases}
            \frac{1-\cos\sqrt{\kappa}t}{\kappa},&\kappa>0;
            \\\frac{t^2}{2},&\kappa=0.
        \end{cases}
    \end{align*}
    Then,
    \begin{enumerate}
        \item If $u_1,u_2\in C^{\infty}((\Omega,g),B_{\rho}(p))\cap L^{\infty}((\bar{\Omega},g),B_{\rho}(p))$ are harmonic maps, then
        $\theta(u_1,u_2)$ satisfies the weak elliptic maximum principle: for any $L>0$,
        $$\sup_{\p\Omega}\theta(u_1,u_2)\leqslant L\implies\sup_{\Omega}\theta(u_1,u_2)\leqslant L.$$
        \item If $u_1,u_2\in C^{\infty}(\Omega\times[0,T],B_{\rho}(p))\cap L^{\infty}(\bar{\Omega}\times[0,T],B_{\rho}(p))$ are harmonic map flows,
        $\theta(u_1,u_2)$ satisfies the weak parabolic maximum principle:
        $$\sup_{(\p\Omega\times[0,T])\cup(\Omega\times\lbrace0\rbrace)}\theta(u_1,u_2)\leqslant L\implies sup_{\Omega\times[0,T]}\theta(u_1,u_2)\leqslant L.$$
        
    \end{enumerate}
\end{prop}
The supremum above should be understood as essential essential supremum.
\begin{rmk}
    The original statement in J\"ager-Kaul \cite{JagerKaul} requires the continuity of $u_i$ to the boundary. However, we note that the stated weaker version of their results still holds. For part (1), they showed in their paper that there is a function $\Phi$ uniformly bounded on $\Omega$ such that $\theta(u_1,u_2)$ is a subsolution to the equation
    $$\operatorname{div}(e^{-\Phi}\nabla  f)=0.$$
    Now, $e^{-\Phi}$ is in $L^{\infty}(\Omega)$ due to the boundedness of $\Phi$. Testing the subsolution by function $(\theta(u_1,u_2)-L)_+$ and the De Giorgi-Nash-Moser iteration gives the desired weak maximum principle. The part (2) holds for a similar reason.
\end{rmk}

Finally, a priori, harmonic maps and harmonic map flows into a regular ball have uniform gradient estimates.
\begin{lem}\label{lem: gradient estimate}
    Fix $E_0>0$. We have the followings.
    \begin{enumerate}
        \item (Hildebrandt-Kaul-Widman \cite{HildebrandtStefanWidman}) If $u\in C^{\infty}(\Omega,B_{\rho}(p))$ is a harmonic map with respect to $g$ with $E_{\Omega}(u)\leqslant E_0$, then for each $\Omega'\Subset\Omega$,
        $$\Vert u\Vert_{C^2(\Omega')}\leqslant C(\Omega,\Omega',g,E_0,\rho,M).$$
        \item If $u\in C^{\infty}(\Omega\times[0,T],B_{\rho}(p))$ is a harmonic map flow with $E_{\Omega}(u(\cdot,0))\leqslant E_0$, then for each $\Omega'\Subset\Omega$ and $0<T'\leqslant T$, 
        $$\Vert u\Vert_{C^2(\Omega'\times[T',T])}\leqslant C(\Omega,\Omega',g,E_0,T',\rho,M).$$
    \end{enumerate}
\end{lem}
\begin{proof}
    For the static case, the proof was already established by Hildebrandt-Kaul-Widman \cite{HildebrandtStefanWidman}. Their proof even implies that any weakly harmonic map into a regular ball is smooth. Here we give a sketch of proof for (2) using the compactness theorems by F. Lin and F. Lin-C. Wang \cite{LinGradientEstimates,LinWang}. 
    To apply the theorems therein, we note that since $h=\dist(p,\cdot)^2$ is a strictly convex function on $B_{\rho}(p)$. Hence, there is no non-constant harmonic map from $S^l$ to $B_{\rho}(p)$ for any $l$. 
    
    To prove (2), assume there is a sequence of smooth harmonic map flows $u_k$ from $\Omega\times[0,T]$ to $B_{\rho}(p)$ which has energy bound $E_{\Omega}(u_k(\cdot,0))\leqslant E_0$ but $$\Vert u_k\Vert_{C^2(\Omega'\times[T',T])}\rightarrow\infty.$$ After passing to a subsequence, by \cite{LinWang} $u_k\rightarrow u$ strongly in $W^{1,2}_{\loc}(\Omega\times[0,T],B_{\rho}(p))$ since there is no harmonic $S^2$ in $B_{\rho}(p)$. $u$ is a smooth harmonic map flow since there is no nonconstant harmonic spheres in $B_{\rho}(p)$. Then $u_k$ converges to $u$ locally smoothly using the small energy regularity in Proposition \ref{prop: monotonocity}. This contradicts the fact that $\Vert u_k\Vert_{C^2(\Omega'\times[T',T])}\rightarrow\infty$. (1) can be proved in a similar manner. 
\end{proof}
We are now ready to show the following existence result of harmonic map flow from arbitrary $W^{1,2}_{loc}$ data.
\begin{prop}\label{prop: existence}
    Given a map $u_0\in W^{1,2}_{\loc}(\R^n,B_{\rho}(p))$, there exists $u\in C^{\infty}(\R^n\times(0,\infty),B_{\rho}(p))$ which is the harmonic map flow starting from $u_0$.
\end{prop}
\begin{proof}
    We first consider the existence on bounded domains with smooth initial data. This part is already done by Jost \cite{Jost}. Namely, fix a bounded smooth domain $\Omega\subset \R^n$ and $u_0\in C^{\infty}(\bar{\Omega})$, we consider the following problem
    \begin{align*}
        \begin{cases}
            \p_t u=\Delta u+A_u(\nabla u,\nabla u),&\text{ in }\Omega\times[0,\infty);\\
            u=u_0&\text{ on }(\Omega\times\lbrace0\rbrace)\cup(\p\Omega\times[0,\infty)).
        \end{cases}
    \end{align*}
    By Galerkin method, such a problem has a smooth solution at least for a small time interval. Denote $[0,T_{\max})$ the maximal interval on which the solution exists. Using the equation of harmonic map flow, we have the following energy identity
    $$\int_{\Omega}\vert\nabla u_0\vert^2\,\dif x-\int_{\Omega}\vert\nabla u(\cdot,T)\vert^2\,\dif x=2\int_0^T\int_{\Omega}\vert\p_t u\vert^2\,\dif x\dif t,$$
    which gives us the energy bound for $u(\cdot, t)$. By a well-known fact (see \cite{SolonnikovE} for example), if $T_{\max}<\infty$, then we must have 
    $$\lim_{t\rightarrow T_{max}^-}\Vert\nabla u(\cdot,t)\Vert_{L^{\infty}(\Omega)}=\infty.$$
    However, this contradicts Lemma \ref{lem: gradient estimate}(2). We deduce that $T_{\max}=\infty$, i.e. the solution exists on $[0,\infty)$.

    For the problem on $\R^n$, when $u_0$ has regularity merely in $W^{1,2}_{\loc}$, since $B_{\rho}(p)$ is homotopically trivial, we can use the approximation results of Bethuel \cite{BethuelApproximation} to get a sequence of maps $(u_0)_k\in C^{\infty}(\R^n,B_{\rho}(p))$ which converges to $u_0$ in $W^{1,2}_{\loc}$. Take a sequence $R_k\rightarrow\infty$, let $u_k\in C^{\infty}(\bar{B}_{R_k}\times[0,\infty),B_{\rho}(p))$ be the solution to the following problems
    \begin{align*}
        \begin{cases}
            \p_t u_k=\Delta u_k+A_{u_k}(\nabla u_k,\nabla u_k)&\text{ in }B_{R_k}\times[0,\infty);\\
            u_k=(u_0)_k,&\text{ on }(B_{R_k}\times\lbrace0\rbrace)\cup(\p B_{R_k}\times(0,\infty)).
        \end{cases}
    \end{align*}
    By Lemma \ref{lem: gradient estimate}(2), we can find a subsequence $u_k\rightarrow u$ locally smoothly in $\R^n\times(0,\infty)$. For each $R>0$,$E_{B_R}(u_k(\cdot,T))$ is uniformly bounded for each $T>0$ and large $k$ by the monotonicity formula. By dominated convergence theorem, $u_k\rightarrow u$ in $W^{1,2}_{\loc}(\R^n\times[0,\infty))$. Hence, we conclude that $u(x,0)=u_0$ in the sense of trace. Moreover, for each $k$ there holds by the equation of harmonic map flow
    \begin{align*}
        \int\vert\nabla u_k(\cdot,T)\vert^2\varphi^2\,\dif x+\int_0^T\int\vert\p_tu_k\vert^2\varphi^2\,\dif x\dif t\leqslant\int\vert\nabla (u_0)_k\vert^2\varphi^2\,\dif x+4\int_0^T\int\vert\nabla u_k\vert^2\vert\nabla \varphi\vert^2\,\dif x\dif t,
    \end{align*}
    for all $T>0$ and $\varphi\in C^{\infty}_c(\R^n)$. Letting $k\rightarrow\infty$, we deduce that
    \begin{align*}
        \int\vert\nabla u(\cdot,T)\vert^2\varphi^2\,\dif x+\int_0^T\int\vert\p_tu\vert^2\varphi^2\,\dif x\dif t\leqslant\int\vert\nabla u_0\vert^2\varphi^2\,\dif x+4\int_0^T\int\vert\nabla u\vert^2\vert\nabla \varphi\vert^2\,\dif x\dif t.
    \end{align*}
    Letting $T\rightarrow0$, we see that for each $R>0$,
    $$\int_0^1\int_{B_R}\vert\p_t u\vert^2\dif x\dif t\leqslant C(E_{B_{2R}}(u_0),R),$$
    $$\limsup_{t\rightarrow0}E_{B_{R}}(u(t))\leqslant E_{B_R}(u_0).$$
    This tells us $u(\cdot,t)\rightarrow u_0$ strongly in $W^{1,2}_{\loc}$ as $t\rightarrow0$.
\end{proof}

\subsection{Energy minimizing maps into a hemisphere}
Due to the uniqueness result, Proposition \ref{prop: uniqueness}, all harmonic maps into a regular ball belong to a smaller class, the energy minimizing maps.
\begin{defn}
    Let $u\in W^{1,2}((\Omega,g),M)$. We say $u$ is \textbf{(locally) energy minimizing}, if for all $\Omega'\Subset\Omega$ and $v\in W^{1,2}((\Omega,g),M)$ with $u=v$ on $\Omega\setminus\Omega'$, there holds
    $$E_{\Omega'}(u)\leqslant E_{\Omega'}(v).$$
\end{defn}

The energy minimizing maps enjoys better convergence scheme than general $W^{1,2}$ maps and have a priori regularity estimate.
\begin{prop}[Schoen-Uhlenbeck\cite{SchoenUhlenbeckRegularity,SchoenUhlenbeckSphere}]\label{prop: compactness}
    We have the following properties for energy minimizing maps.
    \begin{enumerate}
        \item If $u_k\in W^{1,2}(\Omega,M)$ is a sequence of energy minimizing map with respect to $g$ with uniformly bounded energy, then after passing to a subsequence, $u_k$ converges strongly in $W^{1,2}_{\loc}(\Omega,M)$ to an energy minimizing map $u$;
        \item If $u\in W^{1,2}(\Omega,\overline{S^m_+})$ is energy minimizing with respect to $g$, then there exists a closed subset $\Sigma\subset\bar{\Omega}$ which has Hausdorff dimension at most $n-7$ and is isolated when $n=7$, such that $u$ is smooth on $\Omega\setminus\Sigma$. In particular, $u$ is smooth when $n\leqslant 6$.
    \end{enumerate}
\end{prop}

By uniqueness Proposition \ref{prop: uniqueness}, we have that harmonic maps into regular balls are all energy minimizing.

\begin{lem}\label{lem: minimizing property}
    If $u\in C^{\infty}(\Omega,B_{\rho}(p))$ is a harmonic map with respect to $g$, then $u$ is energy minimizing with respect to $g$ among maps in $W^{1,2}\left(\Omega,\bar{B}_{\frac{\pi}{2\sqrt{\kappa}}}(p)\right)$.
\end{lem}
\begin{proof}
    The proof is standard; We recall it here. Take a $\rho'\in\left(\rho,\frac{\pi}{2\sqrt{\kappa}}\right)$. For any $\Omega'\Subset\Omega$, we consider the minimizing problem
    $$\min\lbrace E_{(\Omega',g)}(v):v\in W^{1,2}(\Omega',B_{\rho'}(p)),v=u\text{ on }\p\Omega'\rbrace.$$
    It is easy to see such a minimizer $u'$ always exists and is energy minimizing on $(\Omega',g)$. Moreover, the regularity theory in \cite{HildebrandtStefanWidman} applies to yield the smoothness of $u'$. Then by maximum principle Proposition \ref{prop: uniqueness} applied to regular ball $B_{{(\rho'+\rho)}/2}(p)$, we see $u'=u$ on $\Omega'$. This gives the energy minimizing property of $u$.

    This shows that $u$ is energy minimizing in $W^{1,2}(\Omega',B_{\rho'}(p))$ for each $\rho'\in\left(\rho,\frac{\pi}{2\sqrt{\kappa}}\right)$. For a map  $v\in W^{1,2}\left(\Omega,\bar{B}_{\frac{\pi}{2\sqrt{\kappa}}}(p)\right)$ which agrees with $u$ on $\p\Omega$, we can take a sequence of $\rho_k\rightarrow\frac{\pi}{2\sqrt{\kappa}}$ and $v_k\in W^{1,2}(\Omega',B_{\rho_k}(p))$ which agrees with $u$ on $\p\Omega'$ and $v_k\rightarrow v$ in $W^{1,2}(\Omega,\bar{B}_{\frac{\pi}{2\sqrt{\kappa}}}(p))$. Taking the limit in 
    $$E_{(\Omega',g)}(u)\leqslant E_{(\Omega',g)}(v_k).$$
    We see $u$ is indeed energy minimizing in a larger space $W^{1,2}\left(\Omega,\bar{B}_{\frac{\pi}{2\sqrt{\kappa}}}(p)\right)$.
\end{proof}

\section{Existence of non-trivial expanders}\label{sec: proofs}

We are ready to construct non-trivial expanders of harmonic map flow starting from a given regular 0-homogeneous maps.
\begin{proof}[Proof of Theorem \ref{thm: main theorem 1}]
    Let $u_0:\R^n\rightarrow B_{\rho}(p)$ be a regular 0-homogeneous map. By Proposition \ref{prop: existence}, we see that there exists at least one harmonic map flow starting from $u_0$. 
    \begin{claim}
        Let $u\in C^{\infty}(\R^n\times(0,\infty),B_{\rho}(p))$ be a harmonic map flow starting from $u_0$. Let $\Lambda=\sup_{\p B_1}\vert\nabla u_0\vert$. Then there exists $A=A(n,\kappa,\Lambda)$ such that
        $$\vert\nabla u(x,t)\vert+\sqrt{t}\vert\p_t u(x,t)\vert\leqslant\frac{C(n,\kappa,\Lambda)}{\vert x\vert}\text{ on }\lbrace(x,t):\vert x\vert^2\geqslant A t\rbrace.$$
    \end{claim}
    \begin{proof}[Proof of Claim]
        Rescaling, we only need to prove the claim for $t=1$.
        
        Take an $x_0\in\R^n\setminus\lbrace0\rbrace$ that is large. By the definition of $0$-homogeneous, in particular $u_0(x)=u_0(x/\vert x\vert)$, we have $$\vert\nabla u_0(x)\vert\leqslant\Lambda\vert x\vert^{-1},\text{ for all } x\ne0.$$
        We can estimate
        $$100\int_{B_{10}(x_0)}\vert\nabla u_0\vert^2e^{-\frac{\vert x-x_0\vert^2}{400}}\dif x\leqslant 100\Lambda \cdot10^{n-2}\int_{B_{10}(x_0)}\frac{1}{\vert x\vert^2}e^{-\frac{\vert x-x_0\vert^2}{100}}\dif x\leqslant \frac{C\Lambda}{\vert x_0\vert^2}.$$
        $$\int_{B_{10}(x_0)}\vert\nabla u_0\vert^2\dif x\leqslant\frac{C\Lambda}{\vert x_0\vert^2}$$
        By monotonicity formula Proposition \ref{prop: monotonocity}, we see 
        $$\int_{B_{9}(x_0)\times\lbrace t\rbrace}\vert\nabla u\vert^2\dif x\leqslant\frac{C\Lambda}{\vert x_0\vert^2}\text{ for all }\frac{1}{2}\leqslant t\leqslant2.$$
        Hence, if $\vert x_0\vert$ is sufficiently large, then we can apply the small energy regularity at $(x_0,10)$, to get the gradient estimate
        $$\sup_{B_9(x_0)\times[1/2,2]}(\vert\nabla u\vert^2+\vert\p _t u\vert^2)\leqslant C.$$
        Recall that we have the Bochner type inequality (see \cite[Chapter 7]{LinWangBook}
        $$(\p_t-\Delta)\vert\nabla u\vert^2\leqslant C\vert\nabla u\vert^4\leqslant C\vert\nabla u\vert^2.$$
        $$(\p_t-\Delta)\vert\p_t u\vert^2\leqslant C\vert\nabla u\vert^2\vert\p_t u\vert^2\leqslant C\vert\p_t u\vert^2\text{ on }B_{9}(x_0)\times[1/2,2].$$
        By parabolic Moser's Harnack inequality, together with \cite[Lemma 2.3]{ChenLiWang}, we get
        $$\sup_{B_{1/2}(x_0)\times[3/4,5/4]}\vert\nabla u\vert^2\leqslant C\int_{B_{1/2}(x_0)\times[3/4,5/4]}\vert\nabla u\vert^2\leqslant\frac{C\Lambda}{\vert x_0\vert^2}.$$
        \begin{align*}
            \sup_{B_{1/4}(x_0)\times[15/16,17/16]}\vert\p_t u\vert^2&\leqslant C\int_{B_{1/3}(x_0)\times[8/9,10/9]}\vert\p_t u\vert^2\\&\leqslant C\int_{B_{1/2}(x_0)\times[3/4,5/4]}\vert\nabla u\vert^2\leqslant\frac{C\Lambda}{\vert x_0\vert^2}.
        \end{align*}
        This proves the claim.
    \end{proof}
    The higher ordergradient estimates follows in a similar way using a bootstrap argument. This gives the smoothness of $u$ near $(\R^n\setminus\lbrace0\rbrace)\times\lbrace0\rbrace
    $.
    Hence, for any two harmonic map flows $u,v\in C^{\infty}(\R^n\times(0,\infty))$ starting from $u_0$, for each fixed $T$ we have
    \begin{align*}
        \vert u(x,T)-v(x,T)\vert&\leqslant \vert u(x,T)-u_0(x)\vert+\vert v(x,T)-v_0(x)\vert\\&\leqslant\int_0^T(\vert\p_t u\vert+\vert\p_t v\vert)\dif t\leqslant \frac{CT}{\vert x\vert}, \text{ for }\vert x\vert> A\sqrt T.
    \end{align*}
    By maximum principle Proposition \ref{prop: uniqueness} applied to $\theta(u,v)$ on $B_R\times[0,T]$ for $R>A \sqrt T$, we see
    $$\theta(u,v)\leqslant CTR^{-2}\text{ on } B_R\times[0,T]\text{ for } R\text{ sufficiently large}.$$
    Letting $R\rightarrow\infty$, we get that $u=v$ on $\R^n\times[0,T]$. Since $T$ is arbitrary, we know that $u=v$. This proves the uniqueness part of the theorem.

    Finally, since $u$ and $u_{\lambda}=u(\lambda\cdot,\lambda^2\cdot)$ are both harmonic map flows starting from $u_0$, we see $u=u_{\lambda}$. Hence $u$ is self-similar, and $u(x)=u(x,1)$ is an expander to harmonic map flow.
\end{proof}

\begin{proof}[Proof of Theorem \ref{thm: main theorem 2} and Theorem \ref{thm: main theorem 3}]
    Denote the coordinate on $\overline{S^m_+}$ as $(y',y_{m+1})$ where $y'\in\R^m,y_{m+1}\geqslant0$ and $\vert y'\vert^2+y_{m+1}^2=1$. We henceforth write $u_0=(u_0',(u_0)_{m+1})$. For $\sigma>0$ that is small, let us consider the perturbation of $u_0$ defined by $$u_0^{\sigma}=\frac{(u_0',(u_0)_{m+1}+\sigma)}{\vert (u_0',(u_0)_{m+1}+\sigma)\vert}.$$
    Recall that the standard sphere has sectional curvature 1. Let $p$ denote the north pole of $S^m$, then $B_{\pi/2}(p)=S^m_+$ and $\lbrace (y',y_{m+1})\in S^m:y_{m+1}\geqslant\sigma/\sqrt{1+\sigma^2}\rbrace$ is inside a regular ball for each $\sigma\in(0,1)$. In particular, $u_0^{\sigma}$ is a regular 0-homogeneous map for each $\sigma$ small. By Theorem \ref{thm: main theorem 1}, there exists an expander $u^{\sigma}\in C^{\infty}(\R^n,S^m_+)$, which viewed as a harmonic map flow is the unique harmonic map flow starting from $u_0^{\sigma}$ with $\vert\nabla u^{\sigma}\vert\leqslant C$ outside $B_{A}$ for constants $C,A>0$ independent of $\sigma$. Moreover, $u^{\sigma}$ takes the value in a regular ball of $S^m$.

    As discussed in introduction that the expander $u^{\sigma}$ is a harmonic map from $(\R^n,g)$ to $S^m_+$ for $g=e^{\frac{\vert x\vert^2}{2(n-2)}}\delta$. By Lemma \ref{lem: minimizing property}, for each smooth bounded $\Omega\subset\R^n$, $u^{\sigma}$ is energy minimizing among maps in $W^{1,2}(\Omega,\overline{S^m_+})$ with the same boundary values with respect to $g$. For any other map $v=(v',v_{m+1})\in W^{1,2}(\Omega,S^m)$ with $v=u^{\sigma}$ on $\p\Omega$, we consider the map $\tilde{v}=(v',\vert v_{m+1}\vert)$ which takes value in $\overline{S^m_+}$ and $\tilde{v}=u^{\sigma}$ on $\p\Omega$. We see that
    $$E_{(\Omega,g)}(u^{\sigma})\leqslant E_{(\Omega,g)}(\tilde{v})\leqslant E_{(\Omega,g)}(v).$$
    As a result, $u^{\sigma}$ is energy minimizing in $W^{1,2}(\Omega,S^m)$ with respect to $g$. 

    In order to obtain the uniform local energy bound, we first note that the non-negativity of $u_{m+1}^{\sigma}$ and the equation of $(m+1)$-th component $(u^{\sigma})_{m+1}$ says
    $$\Delta_g(u^{\sigma})_{m+1}=-\vert\nabla^gu^{\sigma}\vert^2(u^{\sigma})_{m+1}\leqslant0.$$
    Hence, $(u^{\sigma})_{m+1}$ is superharmonic. Since $(u^{\sigma})_{m+1}$ is asymptotic to $u^{\sigma}_0$ at $\infty$, the positivity of $(u^{\sigma}_0)_{m+1}$ implies the positivity of $(u^{\sigma})_{m+1}\geqslant\sigma/\sqrt{1+\sigma^2}$. In this situation, we have a positive solution to the equation
    $$\Delta_gf+\vert\nabla^g u^{\sigma}\vert^2f=0.$$
    Classical PDE theory implies that the first eigenvalue of $\Delta_g+\vert\nabla^gu^{\sigma}\vert^2$ is non-negative in any domain. In particular,
    $$\int(\vert\nabla\varphi\vert^2-\vert\nabla u\vert^2\varphi^2)e^{\frac{\vert x\vert^2}{4}}\dif x\geqslant0,\text{ for all }\varphi\in C_c^{\infty}(\R^n).$$
    Fixing $R>0$ and $\varphi$ to be standard cutoff function $\varphi\in C_c^{\infty}(B_{2R})$, such that $\varphi=1$ on $B_R$, we see there is uniform local energy bound for $u^{\sigma}$:
    $$\int_{B_R}\vert\nabla u^{\sigma}\vert^2\dif x\leqslant C(n,R).$$
    
    Hence, we can pass to a subsequence and use Proposition \ref{prop: compactness}(1) to get that $u^{\sigma}$ converges to an expander $u\in W^{1,2}_{\loc}(\R^n,\overline{S^m_+})$ strongly in $W^{1,2}_{\loc}$. $u$ is energy minimizing with respect to $g$. Then by Proposition \ref{prop: compactness}(2), $u$ is smooth when $3\leqslant n\leqslant6$ and smooth away from an $(n-7)$-dimensional closed set when $n\geqslant 7$. The uniform gradient bounds for $u^{\sigma}$ outside $B_A$ tell us $\Sigma$ is bounded. 
    
    To see $u(x,t)=u(x/\sqrt{t})$ has initial value $u_0$, we first appeal to the energy inequality of $u^{\sigma}$
    \begin{align*}
        \int\vert\nabla u^{\sigma}(\cdot,T)\vert^2\varphi^2\,\dif x+\int_0^T\int\vert\p_tu^{\sigma}\vert^2\varphi^2\,\dif x\dif t\leqslant\int\vert\nabla u_0^{\sigma}\vert^2\varphi^2\,\dif x+4\int_0^T\int\vert\nabla u^{\sigma}\vert^2\vert\nabla \varphi\vert^2\,\dif x\dif t,
    \end{align*}
    for all $T>0$ and $\varphi\in C^{\infty}_c(\R^n)$, passing to limit we know that the same inequality also holds for $u$. This first establishes local uniform bound of $\Vert \p_tu^{\sigma}\Vert_{L^2([0,T]\times B_R)}$. Hence, the Aubin-Lions lemma gives the strong convergence of $u^{\sigma}$ to $u$ in $C([0,T],L^2(B_R))$ for each $R>0$. From this we see that $u(\cdot,t)\rightarrow u_0$ as $t\rightarrow0$ strongly in $L^2_{\text{loc}}$. Then the energy inequality above upgrades this to the strong convergence in $W^{1,2}_{\text{loc}}$. We have established that $u$ is the harmonic map flow starting from $u_0$. 
    
    Next, we show that $u$ must differ from $u_0$. When $3\leqslant n\leqslant 6$, this is already seen from the smoothness of $u$. For $n\geqslant 7$, we note that if $u_0$ is not energy minimizing for standard Euclidean metric, then it also cannot be minimizing with respect to $g=e^{\frac{\vert x\vert^2}{2(n-2)}}\delta$ as follows. First by scaling invariance, $u_0$ is not energy minimizing on $B_1$. For any $v\in W^{1,2}(B_1,\overline{S^m_+})$ that agrees with $u_0$ on $\p B_1$ and $r>0$, we know $v^r(y)=v(r^{-1}y)$ is a map in $W^{1,2}(B_r,\overline{S^m_+})$ which agrees with $u_0$ on $\p B_r$. If $u_0$ is energy minimizing with respect to $g$, we deduce that
    $$\int_{B_r}\vert\nabla u_0\vert^2e^{\frac{\vert y\vert^2}{4}}\dif y\leqslant\int_{B_r}\vert\nabla  v^r\vert^2e^{\frac{\vert y\vert^2}{4}}\dif y=r^{-2}\int_{B_r}\vert\nabla v(r^{-1}y)\vert^2e^{\frac{\vert y\vert^2}{4}}\dif y.$$
    Let us consider the change of variable $y=r x$. Since $u_0$ is 0-homogeneous, we see $\nabla u_0(rx)=r^{-1}\nabla u_0(x)$. We get
    $$\int_{B_1}\vert\nabla u_0\vert^2 e^{\frac{r^2\vert x\vert^2}{4}}\dif x\leqslant\int_{B_1}\vert\nabla v\vert^2e^{\frac{r^2\vert x\vert^2}{4}}\dif x.$$Letting $r\rightarrow0$, we see $u_0$ is energy minimizing with respect to the standard Euclidean metric. This contradicts our assumption and tells us $u$ cannot be identically equal to $u_0$.
\end{proof}

\section{Non-uniqueness of harmonic map flow in bounded domains}
Finally, we establish the following concatenation result, which together with Theorem \ref{thm: main theorem 2} implies Corollary \ref{thm: non-uniqueness on bounded domain}.

\begin{prop}\label{prop: concatenation}
    Let $u_0:\R^n\rightarrow S^k$ be a non-minimizing regular 0-homogeneous stationary harmonic map. Suppose there exists an expander $u\in C^{\infty}(\R^n,S^k)$ which can be viewed as a harmonic map flow starting from $u_0$.

    Then the following equation has a non-trivial solution $v\in C^{\infty}(\bar B_1\times(0,T],S^k)$ for some $T>0$.
    \begin{align*}
        \begin{cases}
            \p_t v=\Delta v+\vert \nabla v\vert^2v&\text{ in } B_1\times(0,T];\\
            v(\cdot,t)\rightarrow u_0&\text{ strongly in } W^{1,2}(B_1)\text{ as }t\rightarrow 0;\\
            v(\cdot,t)=u_0&\text{ on }\p B_1\times[0,T].
        \end{cases}
    \end{align*}
\end{prop}
We need the following asymptotic analysis, which is due to Deruelle \cite{Deruelle}. We remark that our assumption that $u(\cdot,t)\rightarrow u_0$ strongly in $W^{1,2}$ as $t\rightarrow0$, together with the use of monotonicity formula and $\epsilon$-regularity as in the proof of the first claim of Theorem \ref{thm: main theorem 1} guarantee that $u$ us regular at infinity in the sense of \cite{Deruelle}. The regular at infinity property of $u_0$ comes directly from the scaling invariance of $u_0$ as well as the smoothness of $u_0|_{\p B_1}$.
\begin{lem}[{Deruelle \cite[Section 2]{Deruelle}}]\label{lem: asymptotic of expander}
    Let $u$ and $u_0$ be as in Proposition \ref{prop: concatenation} and $\Lambda=\sup_{\p B_1}\vert\nabla u_0\vert$. Then there exists a constant $A=A(n,\Lambda)>0$ such that
    $$\vert\nabla^i\p_t^j(u(x,t)-u_0(x))\vert\leqslant C(i,j,n,\Lambda)\vert x\vert^{i+2j-n}t^{(n-i-2j)/2}e^{-\vert x\vert^2/(4t)}\text{, if }\vert x\vert\geqslant A\sqrt t.$$
    In particular, there exists a $T_0=T_0(n,\Lambda)>0$ such that for all $0\leqslant i,j\leqslant10$,
    $$\vert\nabla^i\p_t^j(u(x,t)-u_0(x))\vert\leqslant e^{-c/t}<1/10,\text{ on }(\bar B_{1}\setminus B_{1/3})\times[0,T_0].$$
\end{lem}

\begin{proof}[Proof of Proposition \ref{prop: concatenation}]
    Let us choose a cutoff function $\zeta\in C_c^{\infty}(B_{3/4})$ such that $\zeta\geqslant0$, $\zeta=1$ on $B_{1/2}$ and set
    $$\hat u=\frac{\zeta u+(1-\zeta)u_0}{\vert\zeta u+(1-\zeta)u_0\vert}.$$
    By Lemma \ref{lem: asymptotic of expander}, when $t\leqslant T_0$ and $\vert x\vert\leqslant1$ the denominator in the right hand side is nonzero. Hence, $\hat u\in C^{\infty}(\bar B_1\times(0,T_0],S^k)$. Note that $\hat u=u$ on $\bar B_{1/2}$, $\hat u=u_0$ outside $B_{3/4}$ and $\hat u(\cdot,t)\rightarrow u_0$ strongly in $W^{1,2}(B_1)$ as $t\rightarrow0$. Moreover, Lemma \ref{lem: asymptotic of expander} implies
    $$\vert\mathcal{E}(\hat u)\vert\leqslant e^{-2c_0/t}\text{ on }B_1\times (0,T_0],$$
    where $c_0>0$ is a small constant and $\mathcal{E}(\hat u):=\p_t\hat u-\Delta \hat u-\vert\nabla\hat u\vert^2\hat{u}.$

    We wish to solve a harmonic map flow of the form $v=\hat u+\hat w$ on $(0,T]$ for $T\in(0,T_0]$ and a map $\hat w=0$ on $\p B_1\times(0,T]$ and $\hat w$ vanishes at $t=0$ exponentially. We solve $\hat w$ by means of fixed point argument. To this end, we need the following estimate between Banach spaces.
    \begin{claim}
        For a map $F\in C^{0,1}(\bar B_1\times(0,T],\R^{k+1})$, define the norm
        $$\Vert F\Vert_{X_{T}}=\sup_{0<t\leqslant T}e^{c_0/t}(\Vert F(\cdot,t)\Vert_{L^{\infty}}+\sqrt t\Vert\nabla F(\cdot,t)\Vert_{L^{\infty}}.$$
        And define for $w\in C(\bar B_1\times(0,T],\R^{k+1})$ the norm
        $$\Vert w\Vert_{Y_T}=\sup_{0<t\leqslant T}te^{c_0/t}\Vert w(\cdot,t)\Vert_{L^{\infty}}.$$
        Let us consider the operator
        $$\mathcal{K}w(\cdot,t)=\int_0^te^{(t-s)\Delta_D}w(\cdot,s)\dif s,$$
        where for a given function $g\in C_0(\bar B_1)$, $e^{t\Delta_D}g=:f(\cdot,t)$ denotes the solution to heat equation $\p_tf=\Delta f$ with Dirichlet boundary condition $f(\cdot,t)=0$ on $\p B_1$. We have the estimate
        $$\Vert\mathcal{K}w\Vert_{X_T}\leqslant C\sqrt T\Vert w\Vert_{Y_T},\text{ for small } T.$$
    \end{claim}
    \begin{proof}[Proof of Claim]
        Write $L=\Vert w\Vert_{Y_T}$. We can assume $L<\infty$ otherwise the claim is trivial. Then $\Vert w(\cdot,s)\Vert_{L^{\infty}}\leqslant Ls^{-1}e^{-c_0/s}$ for all $s\in(0,T]$. Recall that the classical Parabolic estimate gives $\Vert e^{s\Delta_D}w\Vert_{L^{\infty}}\leqslant \Vert w\Vert_{L^{\infty}}$ and $\Vert \nabla e^{s\Delta D}w\Vert_{L^{\infty}}\leqslant Cs^{-1/2}\Vert w\Vert_{L^{\infty}}$. Hence we have
        $$e^{c_0/t}\Vert\mathcal{K}w(\cdot,t)\Vert_{L^{\infty}}\leqslant e^{c_0/t}\int_0^tLs^{-1}e^{-c_0/s}\dif s=Le^{c_0/t}\int_{c_0/t}^{\infty}r^{-1}e^{-r}\dif r\leqslant Lt/c_0\leqslant LT/c_0.$$
        And the gradient estimate is given by
        $$\sqrt{t}e^{c_0/t}\Vert\nabla\mathcal{K}w(\cdot,)\Vert_{L^{\infty}}\leqslant\sqrt te^{c_0/t}L\int_0^ts^{-1}(t-s)^{-1/2}e^{-c_0/s}\dif s.$$
        We split $$\int_0^{t/2}s^{-1}(t-s)^{-1/2}e^{-c_0/s}\dif s\leqslant \sqrt{2}t^{-1/2}\int_0^{t/2}s^{-1}e^{-c_0/s}\dif s\leqslant\sqrt{2}t^{1/2}e^{-c_0/t}/c_0,$$
        while for $t/2\leqslant s\leqslant t$, we use the identity $1/s-1/t=(t-s)/(st)\geqslant(t-s)/t^2$ to derive
        $$\int_{t/2}^ts^{-1}(t-s)^{-1/2}e^{(c_0/t-c_0/s)}\dif s\leqslant2/t\int_{t/2}^t(t-s)^{-1/2}e^{-c_0(t-s)/t^2}\dif s\leqslant C.$$
        In summary we get 
        $$\sqrt{t}e^{c_0/t}\Vert\nabla\mathcal{K}w(\cdot,t)\Vert_{L^{\infty}}\leqslant CL\sqrt{t}\leqslant CL\sqrt T,$$
        as desired.
    \end{proof}
    Let us continue the proof of Proposition \ref{prop: concatenation}. If $v=\hat u+\hat w$ is a harmonic map flow, then $w$ satisfies the equation
    $$\p_t \hat w=\Delta \hat w+\mathcal{F}(\hat w), \text{ where }\mathcal{F}(\hat w)=-\mathcal{E}(\hat u)+\vert\nabla\hat u\vert^2\hat w+(\vert\nabla\hat w\vert^2+2\langle\nabla\hat u,\nabla \hat w\rangle)(\hat u+\hat w).$$
    Equivalently,
    $$\hat w(\cdot,t)=\int_0^te^{(t-s)\Delta_D}\mathcal{F}(\hat w)(\cdot,s)\dif s=\mathcal{K}\mathcal{F}(\hat w),\hat w(\cdot,0)=0=\hat w(\cdot,t)|_{\p B_1}\text{ for }0\leqslant t\leqslant T.$$
    Let us work within the class $\Vert w\Vert_{X_T}\leqslant L\leqslant1$. Such a $w$ satisfies $\Vert w(\cdot,t)\Vert_{L^{\infty}}\leqslant e^{-c_0/t}L$ and $\Vert \nabla w(\cdot,t)\Vert_{L^{\infty}}\leqslant e^{-c_0/t}t^{-1/2}L$. By Lemma \ref{lem: asymptotic of expander}, there hold $\vert\nabla\hat u\vert\leqslant Ct^{-1/2}$. We can then estimate the terms in $\mathcal{F}(w)$ by
    \begin{align*}
        &\vert\nabla\hat u\vert^2w+(\vert\nabla w\vert^2+2\langle\nabla\hat u,\nabla w\rangle)(\hat u+w)\\\leqslant &Ct^{-1}e^{-c_0/t}(L+L^2e^{-c_0/t}(1+Le^{-c_0/t})+Le^{-c_0/t}(1+Le^{-c_0/t}))\\\leqslant &Ct^{-1}e^{-c_0/t}L.
    \end{align*}
    And recall $\vert\mathcal{E}(\hat u)\vert\leqslant e^{-2c_0/t}\leqslant e^{-2c_0/T}$. Therefore,
    $$\Vert \mathcal{F}(w)\Vert_{Y_T}\leqslant Te^{-c_0/T}+CL.$$
    By our Claim, $$\Vert\mathcal{K}\mathcal{F}(w)\Vert_{X_T}\leqslant C\sqrt T\Vert\mathcal{F}(w)\Vert_{Y_T}\leqslant C\sqrt T(Te^{-c_0/T}+CL).$$
    Similarly, we have the estimate
    $$\Vert\mathcal{K}\mathcal{F}(w_1)-\mathcal{K}\mathcal{F}(w_2)\Vert_{X_T}\leqslant C\sqrt T(C+L)\Vert w_1-w_2\Vert_{X_T}.$$
    As such, by choosing $T$ so small that $C\sqrt{T}\leqslant1/3,Te^{-c_0/T}\leqslant1/3$ and fixing $L=Te^{-c_0/T}$, we see that $w\mapsto\mathcal{K}\mathcal{F}(w)$ is a contractive self map from the $L$-ball in space $X_T:=\lbrace w\in C((0,T],C^{0,1}_0(\bar B_1,\R^{k+1})):\Vert w\Vert_{X_T}<\infty\rbrace.$ It is easy to see that $X_T$ is a Banach space. Hence, $w\mapsto\mathcal{K}\mathcal{F}(w)$ admits a unique fixed point $\hat{w}$ in this $L$-ball. By definition $v=\hat u+\hat w$ is a harmonic map flow starting from $u_0$. The smoothness of $v$ on $(0,T]$ is guaranteed by $\Vert \hat{w}\Vert_{X_T}\leqslant L$ and standard parabolic estimates. $\Vert w(\cdot,t)\Vert_{L^{\infty}}+\sqrt t\Vert\nabla w(\cdot,t)\Vert_{L^{\infty}}\leqslant e^{-c_0/t}$ guarantees that $w(\cdot,t)\rightarrow0$ strongly in $W^{1,2}(B_1)$, which implies $v(\cdot,t)\rightarrow u_0$ strongly in $W^{1,2}(B_1)$ as $t\rightarrow0$. The smoothness of $v$ tells us that $v$ is non-trivial, because the initial value $u_0$ is singular at $0$.

    To see $v$ is sphere valued, set $f=1-\vert v\vert^2$. Then $\p_t f-\Delta f=2\vert\nabla v\vert^2f$ in $B_1\times(0,T],f=0$ on $\p B_1\times(0,T]$. The gradient bounds for $\hat u$ and $w$ implies $\Vert\nabla v(\cdot,t)\Vert_{L^{\infty}}\leqslant Ct^{-1/2}$. Moreover, $\hat w\in X_T$ tells us that $\Vert \hat w(\cdot,t)\Vert_{L^{\infty}}\leqslant e^{-c_0/t}$, which together with $\vert\hat u\vert=1$ implies $\vert f\vert=\vert2\langle\hat u,\hat{ w}\rangle+\vert\hat w\vert^2\vert\leqslant 3e^{-c_0/t}$. Parabolic maximum principle gives on any $[\delta,t]$,
    $$\Vert f(\cdot,t)\Vert_{L^{\infty}}\leqslant\Vert f(\cdot,\delta)\Vert_{L^{\infty}}\exp\left(\int_{\delta}^t\frac{C}{s}\dif s\right)\leqslant3e^{-c_0/\delta}\delta^{-C}.$$
    Taking $\delta\rightarrow0$, we derive that $f=0$ for all $t$, which is equivalent to saying $v$ is a map to $S^k$. 
\end{proof}

\bibliography{reference}
\bibliographystyle{plain}
\end{document}